\documentclass{article}
\usepackage{xcolor}
\usepackage{hyperref}
\hypersetup{
  colorlinks   = true, 
  urlcolor     = blue, 
  linkcolor    = green!50!black, 
  citecolor   = red 
}
\usepackage[font=small,labelfont=bf]{caption}
\usepackage{import}
\usepackage{pdfpages}
\usepackage{transparent}

\usepackage{multirow}

\usepackage{authblk}
\usepackage{caption, subcaption}

\usepackage{mathtools}
\mathtoolsset{showonlyrefs=true}

\DeclareCaptionLabelFormat{subfigure}{\thefigure#2: }
\usepackage{tikz}
\usetikzlibrary {calc}

\usepackage{geometry}

\newcommand{\rev}[1]{{\textcolor{black}{#1}}}

\usepackage{todonotes}
\usepackage{csquotes}

\usepackage{orcidlink}

\usepackage{amsmath,amsfonts,amsthm,amssymb}
\usepackage{bm}

\newtheorem{Definition}{Definition}
\newtheorem{Remark}{Remark}
\newtheorem{Proposition}{Proposition}

\newtheorem{Theorem}{Theorem}
\newtheorem{Lemma}{Lemma}

\title{Networks of Pendula with Diffusive Interactions}

\author[$1$]{Riccardo Bonetto\orcidlink{0000-0001-8075-6147}\footnote{r.bonetto@rug.nl}}
\author[$1$]{Hildeberto Jardón-Kojakhmetov\orcidlink{0000-0001-8708-7409}\footnote{h.jardon.kojakhmetov@rug.nl}}
\author[$2$]{Christian Kuehn\orcidlink{0000-0002-7063-6173}\footnote{ckuehn@ma.tum.de}}

\begin{document}
\affil[$1$]{University of Groningen — Bernoulli Institute for Mathematics, Computer Science and Artificial Intelligence; Groningen, The Netherlands}
\affil[$2$]{Faculty of Mathematics, Technical University of Munich; Garching, Germany.}
\maketitle

\abstract{We study a system of coupled pendula with diffusive interactions, which could depend both on positions and on momenta. The coupling structure is defined by an undirected network, while the dynamic equations are derived from a Hamiltonian; as such, the energy is conserved. We provide sufficient conditions on the energy for bounded motion, identifying the oscillatory regime. We also describe the bifurcations of synchronised states, exploiting the quantities related to the network structure. Moreover, we consider patterns of anti-synchrony arising from the specific properties of the model. Such patterns can be linked to global properties of the system \rev{by looking at motion of the centre of mass}. Nonetheless, the anti-synchrony patterns play also a relevant role in the bifurcation analysis. For the case of graphs with a large number of nodes, we characterise the parameter range in which the bifurcations occur. We complement the analysis with some numerical simulations showing the interplay between bifurcations of the origin and transitions to chaos of nearby orbits. A key feature is that the observed chaotic regime emerges at low energies.}

\section{Introduction}

Coupled oscillators have been extensively studied for their prominence in applied sciences, and for their mathematical appeal \cite{ashwin1992dynamics, sepulchre1997localized}. Indeed, oscillatory motion is ubiquitous in nature. Systems showing complex oscillatory behaviour include neuronal networks \cite{ermentrout2002modeling, schwemmer2012theory, buzsaki2004neuronal}, biological systems \cite{strogatz1993coupled}, chemical oscillations \cite{Kuramoto2003}, among many others. Of course, physical systems play a relevant role in the picture, for example, Josephson junctions \cite{levi1978dynamics, wiesenfeld1998frequency} are modeled as coupled pendula, which are paradigmatic examples of nonlinear oscillators \cite{henderson1991geometry, huynh2010two, levi1988caterpillar}. It is well known that dissipative systems can show limit cycles, and in turn the dynamics of limit cycle oscillators can often be reduced to phase oscillators \cite{golomb1992clustering, restrepo2006synchronization, ashwin2015weak}. On the other hand, (planar) conservative systems cannot show limit cycles, and therefore different dynamics and techniques are needed. In the last years, coupled Hamiltonian systems have attracted more attention \cite{aguiar2018coupled, aguiar2019gradient, buono2015dynamics, chan2017topology, tourigny2017networks}. The subject allows for several applications, for example, complex quantum systems \cite{bastidas2015quantum, manzano2013synchronization, nokkala2023complex}, mechanical and biological models \cite{kovaleva2019nonstationary}, and neuromorphic engineering  \cite{njitacke2022hamiltonian}.

In the present paper, we focus our attention on models of $N\geq2$ coupled pendula conserving the energy of the system. In other words, the system is defined by a Hamiltonian, $H=H_0 + \kappa H_1$, where $H_0$ is related to the uncoupled system, and $H_1$ defines the interaction. The coupling is diffusive, and the topology is fixed by a network structure. It is possible to distinguish two main regimes for a pendulum: oscillatory and rotatory, the first one is associated to bounded motion, the later is potentially unbounded. For the coupled system, we give sufficient conditions for bounded motion in terms of the total energy of the system, see Proposition \ref{prop:bounded_motion}. In addition, and thanks to the symmetry properties of the system, it is possible to study patterns of anti-synchrony by means of partitions of the nodes of the network. So, we propose \rev{the centre of mass as a measure of} how `close' the system is to such patterns. In particular, we show that by increasing the coupling strength the system behaves more regularly close to a pattern (convergence is forbidden by the conservation of energy), see Proposition \ref{prop:frequency}. For fully synchronised states we show that their local behaviour is given by imaginary or real eigenvalues, Proposition \ref{prop:synch_eig}, forbidding Hamiltonian-Hopf bifurcations. Proposition \ref{prop:odd_iff} establishes a connection between anti-synchrony spaces and some eigenvectors of the Laplacian, we use this result to show that the bifurcations of the origin restricted to such anti-synchrony spaces are only pitchforks, see Proposition \ref{prop:pitchfork}. We provide an upper bound for the parameter range in which the bifurcations occurs, Proposition \ref{prop:intervals}, discussing the consequences for graphs with a large number of nodes. We conclude our study with some simulations highlighting the interplay between the bifurcations of the origin, and therefore a change in the eigenvalues, to transitions of nearby orbits from regular to chaotic, see Section \ref{sec:simulations}.\medskip


\section{Energy Preserving Coupled Pendula}

We consider network structures given by simple graphs $\mathcal{G}=\{\mathcal{V}, \mathcal{E} \}$, where $\mathcal{V}=\{1, \dots, N\}$ is the set of nodes and $\mathcal{E}$ is the set of edges with elements $\{i,j\}$ representing a connection between the node $i$ and the node $j$. Since we are going to study Hamiltonian systems, the choice of a simple graph structure is not restrictive, as Hamiltonian systems admit only undirected network structures \cite{chan2017topology}. The Hamiltonian defining the system is
\begin{equation}\label{eq:H__N_pendula}
     H=  \sum_{i\in \mathcal{V}} \frac{p_i^2}{2} - \cos{q_i}   + \kappa \sum_{\{i, j\} \in \mathcal{E}}  G(q_i-q_j,p_i-p_j) ,
\end{equation}
where $\kappa$ is a positive parameter associated to the `strength' of interaction. We consider analytic interactions of the form 
\begin{equation}\label{eq:analytical_expansion}
    G(x, y)=\sum_{l,m=0}^\infty c_{lm} x^{2l} y^{2m} ,
\end{equation}
where $c_{lm} \in \mathbb{R}$. We require $H$ to be bounded from below, then without loss of generality we also assume $G\geq0$. \rev{Networked systems are well-known to have symmetry properites related to the graph structure. The standard notion of symmetry in dynamical systems is \emph{equivariance} \cite{golubitsky2006nonlinear, golubitsky2012singularities}, meaning that there exist a group representation acting on the phase space commuting with the vector field. For Hamiltonian systems, it is possible to recognise the symmetries directly from the Hamiltonian function. In such a case, the symmetries appear as the group representation acting on the phase space leaving the Hamiltonian function \emph{invariant}, i.e., with the same form.}


\begin{Lemma}\label{lm:symmetry}
    \rev{Let $\text{Aut}(\mathcal{G})$ be the automorphism group of the graph $\mathcal{G}$, i.e., the set of permutations preserving the adjacency structure of $\mathcal{G}$.} The Hamiltonian \eqref{eq:H__N_pendula} is $\mathbb{Z}_2 \times \text{Aut}(\mathcal{G})$ invariant. Therefore, the vector field induced by $H$ is $\mathbb{Z}_2 \times \text{Aut}(\mathcal{G})$ equivariant.
\end{Lemma}

\begin{proof}
The statements of Lemma \ref{lm:symmetry} are straightforward to check by direct inspection of the equations.
\end{proof}

\begin{Remark}
    \rev{The representation of $\mathbb{Z}_2 \times \text{Aut}(\mathcal{G})$ we are considering in Lemma \ref{lm:symmetry} is given by the permutation representation of $\text{Aut}(\mathcal{G})$ acting on the conjugate variables $(q_i,p_i)$, $i=1, \dots, N$, and $\mathbb{Z}_2$ acting as reflections of the conjugate variables. Clearly $\text{Aut}(\mathcal{G})$ comes form the graph structure describing the interactions of the system, while $\mathbb{Z}_2$ originates from the even nature of the Hamiltonian, or equivalently from the odd nature of the vector field.}
\end{Remark}

The vector field induced by the Hamiltonian is described by the equations
\begin{equation}\label{eq:equations_motion}
    \begin{aligned}
        \dot{q}_i &= p_i + \kappa \sum_{j=1}^N a_{ij} G_{01}(q_i-q_j,p_i-p_j) ,\\
        \dot{p}_i &= - \sin{q_i} - \kappa \sum_{j=1}^N a_{ij} G_{10}(q_i-q_j,p_i-p_j) ,
    \end{aligned} 
\end{equation}
where $i\in \mathcal{V}$, $a_{ij}$ are the elements of the adjacency matrix of $\mathcal{G}$, and we use throughout the text the notation $G_{ij}(x,y)=\frac{\partial^{i+j}G}{\partial x^i\partial y^j}(x,y)$ for the derivatives of the interaction. 

\begin{Proposition}\label{prop:bounded_motion}
     Let $N\geq 2$, then for $H \leq 2-N$ the orbits of \eqref{eq:equations_motion} are bounded. 
\end{Proposition}

\begin{proof}
    Since the kinetic energy and the interaction are non-negative, then momenta are always bounded. We are left to obtain the conditions under which the positions are also bounded.
    Let us rewrite the Hamiltonian \eqref{eq:H__N_pendula} as
    \begin{equation}
        H_\kappa=K - \sum_{i=1}^N \cos{q_i} + \kappa G ,
    \end{equation}
    where $K$ is the kinetic energy and $G$ \rev{is} the interaction. Now, let $\Sigma_\kappa=\{ q \in \mathbb{R}^N \ | \ H_\kappa \leq 2-N\}$ be the position domain defined the the inequality under consideration. Since the interaction is non-negative we have that the decoupled Hamiltonian
    \begin{equation}
        H_0=K - \sum_{i=1}^N \cos{q_i} ,
    \end{equation}
    defines a region $\Sigma_0=\{ q \in \mathbb{R}^N \ | \ H_0 \leq 2-N\}$ that contains $\Sigma_\kappa$. 
    We can rewrite the inequality for $\Sigma_0$ as follows 
    \begin{equation}
        K +(N-2) \leq  \sum_{i=1}^N \cos{q_i} ,
    \end{equation}
    and given the fact that $K\geq0$ we have 
    \begin{equation}\label{eq:ineq}
         \sum_{i=1}^N \cos{q_i} \geq N-2 .
    \end{equation}
    Due to the periodicity of the cosine function, we restrict ourselves to the $N$-dimensional cube $[-\pi, \pi]^N$. The region identified by \eqref{eq:ineq} in the hypercube $[-\pi, \pi]^N$ repeats periodically in $\mathbb{R}^N$. Therefore, only by crossing one of the facets of the cube is possible to have unbounded motion. So, let us consider \eqref{eq:ineq} restricted to the facets of $[-\pi, \pi]^N$. These two objects intersect at points with coordinates
    \begin{equation}\label{eq:points}
        q=\{ (\pm \pi,0, \dots, 0), (0, \pm \pi, 0, \dots, 0), \dots, (0, \dots, 0, \pm \pi) \},
    \end{equation}
    which are exactly $2N$ points intersecting the $2N$ facets of the hypercube, see Figure \ref{fig:q_domain3d}. 
%
    Let us now recall that $\Sigma_0 \supseteq \Sigma_\kappa$, so we need to prove that the points \eqref{eq:points} cannot be crossed by the coupled system under the condition $H_\kappa \leq 2-N$. Restricting $H_\kappa$ to the points \eqref{eq:points} we obtain the inequality $K + \kappa G \leq 0$. Consequently, we obtain the conditions $K=0$ and $G=0$, as they are both non-negative. We recall that $K=0$ implies that the momenta are equal to zero. So, all the terms of the interaction depending on the momenta are zero. \rev{Since $G$ restricted to the points \eqref{eq:points} must be zero, and it is a non-negative function, we have that \eqref{eq:points} are critical points of $G$, and therefore critical points of the vector field. So the motion is bounded as the points \eqref{eq:points} cannot be crossed.}
\end{proof}

\begin{figure}[ht]
\centering
\includegraphics[width=0.5\textwidth]{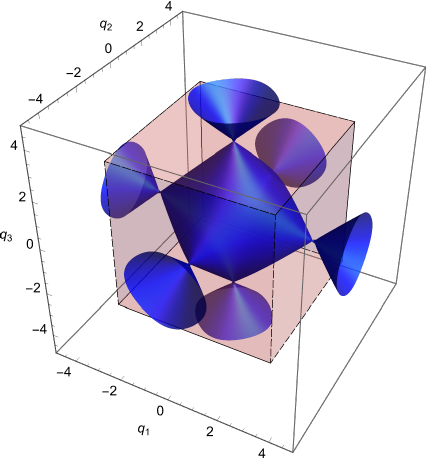}
\caption{The blue surface is defined by the equality $\cos{q_1}+\cos{q_2}+\cos{q_3}=1$, the pink cube is centered at zero and has a side of lenght $2\pi$. We can clearly see that the cube and the surface intersect in six points, one on each side of the cube. The only possibility to have unbounded motion at this energy ($E=-1$) is to cross one of those points. }
\label{fig:q_domain3d}
\end{figure}


\section{Patterns, Invariant Spaces, and Bifurcations}


Studying the dynamics of coupled systems can be very challenging. However, for networked systems one can resort to specific techniques interfacing graph properties to dynamics. The formalism of Golubitsky and Stewart \cite{golubitsky2006nonlinear} allows to find invariant spaces of the dynamical system by looking at balanced partitions of the nodes of the graph. Such relations are particularly relevant as they represent patterns of synchrony of the system. Synchronisation is a key feature of complex systems that has successfully linked mathematical properties to real-world applications \cite{arenas2008synchronization, fell2011role, abarbanel1996synchronisation}. More recently also anti-synchrony patterns have got more attention, especially in the context of chaotic systems \cite{kim2003anti, liu2006antisynchronization, meng2007robust, wedekind2001experimental}.

In this section, by studying \rev{the centre of mass}, we show how global variables can give useful information about anti-synchrony patterns. Later, we study synchrony by explicitly computing its stability. We also study the bifurcations of the origin, where a connection with anti-synchrony patterns is exploited.

\subsection{Anti-Synchrony and \rev{Centre of Mass}}

Networked systems, under some conditions, can exhibit anti-synchrony patterns where a group of nodes have exactly the opposite state of another group at each instant in time. In order to study such patterns, we can look at specific partitions of the set of nodes of the graph that are in correspondence with invariant spaces of the system \cite{neuberger2019synchrony, nijholt2023invariant}. The vector field \eqref{eq:equations_motion} belongs to the class of \emph{odd-difference-coupled} vector fields for which \emph{odd-balanced matched partitions} provide the conditions for anti-synchrony patterns \cite{neuberger2019synchrony}. Let us briefly recall the main definitions and results for odd-balanced matched partitions.
%
We consider \emph{odd partitions} (i.e., containing an odd number of pairwise disjoint subsets) of the set $\mathcal{V}$, where the subsets, $\mathcal{W}$, define equivalence classes for the nodes. For sake of brevity we use the same symbol for the subset and for its equivalence class; also, the equivalence class of a specific node $i$ is denoted by $[i]$.
\rev{A matching function, $m$, is a function acting on the elements, $\mathcal{W}$, of an odd partition such that $m^{-1}=m$, $m$ has only one fixed point, denoted by $\mathcal{W}_0$, and if the empty set, $\emptyset$, is an element of the partition, then $m(\emptyset)=\emptyset$. Roughly speaking, we can think that to each subset of the partition we assign a letter and a sign representing the equivalence class, the matching function acts as a change of sign, $-\mathcal{W} :=m(\mathcal{W})$.}

\begin{Definition}
    A matched partition is an odd partition with a matching function.
\end{Definition}

\begin{Definition}
    Let $\Tilde{\mathcal{V}}$ be a matched partition of $\mathcal{V}$, the linear subspace associated to the matched partition is defined as
    \begin{equation}
        \mathcal{A}:= \{ (q,p) \ | \ (q_i,p_i) = - (q_j,p_j) \ \text{if} \ [i]=-[j] \} ,
    \end{equation}
where the square brackets represent equivalence classes for the nodes.
\end{Definition}

\begin{Definition}
    The degree $d_{\mathcal{W}}(i)$ is the number of edges connecting nodes in $\mathcal{W}$ to $i$.
\end{Definition}

\begin{Definition}
    An odd-balanced partition is a matched partition $\Tilde{\mathcal{V}}$ such that 
    \begin{itemize}
        \item $d_{\mathcal{W}}(i)=d_{-\mathcal{W}}(j)$ whenever $[i]\neq\mathcal{W}_0$ and $\mathcal{W}\neq [i] = -[j]$,
        \item $d_{\mathcal{W}}(i)=d_{-\mathcal{W}}(i)$ whenever $[i]=\mathcal{W}_0\neq \mathcal{W}$,
    \end{itemize}
$\forall \mathcal{W} \in \Tilde{\mathcal{V}}$, and $\forall i,j \in \mathcal{V}$.
\end{Definition}

See Figure \ref{fig:patterns} for some examples.

\begin{Theorem}[Neuberger, Sieben and Swift  \cite{neuberger2019synchrony}]\label{thm:odd}
Let $\Tilde{\mathcal{V}}$ be a matched partition of $\mathcal{V}$. The linear subspace $\mathcal{A}$ is invariant under the flow of \eqref{eq:equations_motion} if and only if $\Tilde{\mathcal{V}}$ is an odd-balanced matched partiton.
\end{Theorem}

\begin{figure}[ht]
\centering
\includegraphics[width=0.7\textwidth]{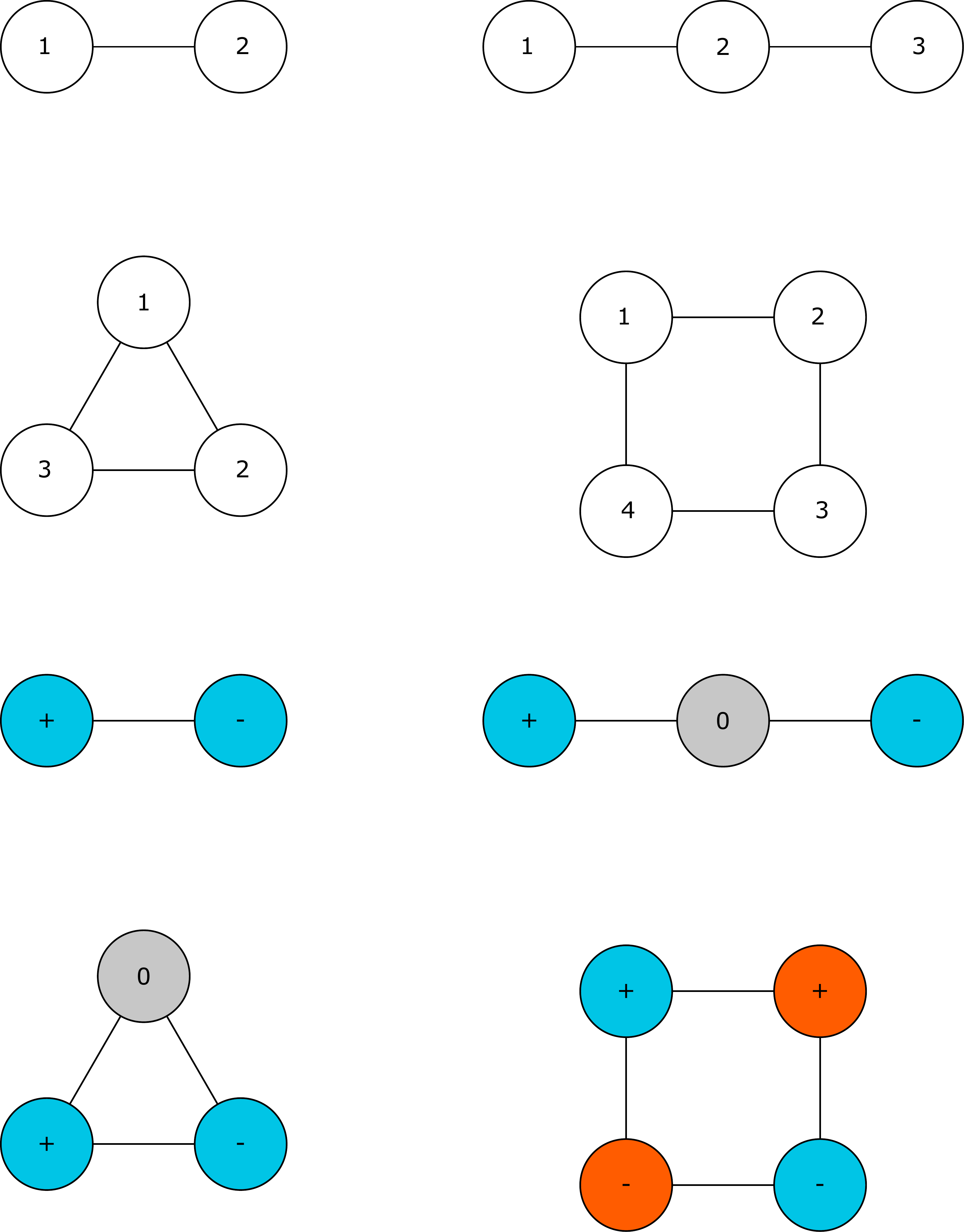}
\caption{We show four examples of odd-balanced partitions. In the upper part of the figure, we show the graphs with labelled vertices. In the lower part, we display a representation of the odd-balanced partitions. Different colours represent different equivalence classes up to a sign that appear as a label of the nodes. The special class fixed by the matching function is represented by the colour gray and the label $0$. For sake of clarity, we write down explicitly the odd-balanced partitions showed. Path graph with two vertices: $\Tilde{\mathcal{V}} = \{  \mathcal{W}_0=\{\}, \mathcal{W}_+=\{ 1\}, \mathcal{W}_-=\{ 2\}   \}$, path graph with three vertices: $\Tilde{\mathcal{V}} = \{  \mathcal{W}_0=\{2\}, \mathcal{W}_+=\{ 1\}, \mathcal{W}_-=\{ 3\}   \}$, complete graph with three vertices $\Tilde{\mathcal{V}} = \{  \mathcal{W}_0=\{1\}, \mathcal{W}_+=\{ 3\}, \mathcal{W}_-=\{ 2\}   \}$, cycle graph with four vertices: $\Tilde{\mathcal{V}} = \{  \mathcal{W}_0=\{\}, \mathcal{W}^\text{blue}_+=\{ 1 \}, \mathcal{W}^\text{orange}_+=\{ 2 \}, \mathcal{W}^\text{blue}_-=\{ 3\}, \mathcal{W}^\text{orange}_-=\{ 4 \}  \}$.
}
\label{fig:patterns}
\end{figure}


If the system is evolving in an invariant space $\mathcal{A}$ we say that it is in an anti-synchrony pattern. It is often useful to have global variables giving information about how `close' the system is to some specific regime, in this case anti-synchrony. We notice that since the coupled system is still conservative, it cannot converge to anti-synchrony patterns. However, we can measure how regularly and close to the pattern the system is evolving.

We define the \rev{equations for the centre of mass} as the two-dimensional system described by the coordinates $(\Bar{q},\Bar{p})$, where 
\begin{align}\label{eq:bar_qp}
    \Bar{q}&:=\frac{1}{N}\sum_{i=1}^N q_i , \\
    \Bar{p}&:=\frac{1}{N}\sum_{i=1}^N p_i .
\end{align}
We can derive the equations for the \rev{centre of mass} by summing up the components of \eqref{eq:equations_motion}. We notice that the tensor $a_{ij} G_{01/10}(q_i-q_j, p_i-p_j)$ is traceless and anti-symmetric, therefore $\sum_{i,j=1}^N a_{ij} G_{01/10}(q_i-q_j, p_i-p_j) = 0$. So, the equations of motion are
\begin{equation}\label{eq:average_system}
        \begin{aligned}
            \dot{\Bar{q}} &= \Bar{p} , \\
            \dot{\Bar{p}} &= - \frac{1}{N} \sum_{i=1}^N \sin{q_i} .
        \end{aligned}
    \end{equation}

\begin{Proposition}
\rev{
    Let $\mathcal{A}$ be an anti-synchrony space associated to an odd-balanced matched partition. The configuration induced by $\mathcal{A}$ is an equilibrium \rev{for the dynamics of the centre of mass} \eqref{eq:average_system}, i.e., $(\Bar{q},\Bar{p})|_\mathcal{A} = 0$ for all $t$.
    }
\end{Proposition}


\begin{proof}
    \rev{It is straightforward to see that the nodes with zero state do not contribute to expression \eqref{eq:bar_qp} for $(\Bar{q},\Bar{p})$. The nodes we are left to consider are the ones that have non-trivial states. By the definition of odd-balanced partition, we have that if $n$ nodes are in the state $(q(t), p(t))$, then another $n$ nodes are in the state $(-q(t), -p(t))$ at all times. Then the sum of equal and opposite terms will cancel out.}
\end{proof}

\begin{Remark}
    We can rewrite \eqref{eq:average_system} as a second-order equation,
    \begin{equation}\label{eq:average_second_order}
        \ddot{\Bar{q}} = - \frac{1}{N} \sum_{i=1}^N \sin{q_i} ,
    \end{equation}
    only depending on the position coordinates, which in principle implicitly depend on $\kappa$.
\end{Remark}

So, we can use the \rev{centre of mass} to quantify how close the system is to anti-synchrony patterns by measuring how close to zero the evolution of $(\Bar{q},\Bar{p})$ is. We expect that the system will stay close to a pattern if the initial conditions are chosen close to regions of $\mathcal{A}$ that are orthogonally stable. However, the motion can still be complicated. We can measure the complexity of the motion close to a pattern by performing some Fourier analysis. 
\rev{From a practical point of view, we could simulate the system and then apply a Fourier transform to the time series. In such a way, we obtain a spectrum of frequencies representing the solution of the equations. The distribution and the amount of different frequencies reveal how complicated the motions is. The simplest periodic function is characterised by one frequency. More complicated periodic function, and also quasi-periodic functions, could show a large range of frequencies, revealing a higher complexity. The following proposition shows that the system regularises by increasing $\kappa$. In other words, for large coupling strength a simple periodic behaviour dominates the dynamics.}


\begin{Proposition}\label{prop:frequency}
    \rev{Let $\mu:=1/\kappa$. For $\mu \ll 1$ be sufficiently small the evolution of the centre of mass is equivalent to a perturbed harmonic oscillator with natural frequency $\omega=1$, i.e.  
    \begin{equation}
        \begin{aligned}
            \dot{\Bar{Q}} &= \Bar{P} , \\
            \dot{\Bar{P}} &= - \Bar{Q} + O\left(\mu\right) .
        \end{aligned}
\end{equation}
    }
\end{Proposition}


\begin{proof}
    Let us start by looking at equations \eqref{eq:equations_motion}, we perform the coordinate rescaling $Q_i:=\kappa q_i$, $P_i:=\kappa p_i$. Then, taking into account the definition $\mu:=1/\kappa$, the vector field becomes
\begin{equation}\label{eq:system_new}
    \begin{aligned}
       \mu \dot{Q}_i  &= \mu P_i  + \frac{1}{\mu} \sum_{j=1}^N a_{ij} G_{01}\left(\mu \left(Q_i-Q_j\right),\mu \left(P_i-P_j\right)\right) , \\
       \mu \dot{P}_i  &= - \sin{\left(\mu Q_i\right)} - \frac{1}{\mu} \sum_{j=1}^N a_{ij} G_{10}\left(\mu \left(Q_i-Q_j\right),\mu \left(P_i-P_j\right)\right) .
    \end{aligned} 
\end{equation}
Now, let us take the average of \eqref{eq:system_new},
\begin{equation}
        \begin{aligned}
            \dot{\Bar{Q}} &= \Bar{P} , \\
            \dot{\Bar{P}} &= -  \frac{1}{\mu N} \sum_{i=1}^N \sin{\left(\mu Q_i\right)} ,
        \end{aligned}
    \end{equation}
where $\Bar{Q}:=\frac{1}{N}\sum_{i=1}^N Q_i$, and $\Bar{P}:=\frac{1}{N}\sum_{i=1}^N P_i$. For $\mu \ll 1$ we can expand the sine function obtaining
\begin{equation}
        \begin{aligned}
            \dot{\Bar{Q}} &= \Bar{P} , \\
            \dot{\Bar{P}} &= - \Bar{Q} + O\left(\mu\right) .
        \end{aligned}
\end{equation}
So, up to order $\mu=1/\kappa$ the system behaves as an harmonic oscillator with frequency $\omega=1$. Notice that the coordinate transformation we performed does not affect the frequency of the original system but only the amplitude.
\end{proof}
A word of care about the averaging process \rev{that leads to the centre of mass coordinates}. We notice that a large coupling strength is a \emph{sufficient} condition for regular behaviour. However, regularity could be explained also by other factors, e.g., evolution sufficiently close to a elliptic equilibrium. Also, since the statement of Proposition \ref{prop:frequency} is quantitative, small but qualitatively relevant features, such as weak chaos, might be not measurable by means of the \rev{centre of mass motion}.

The last part of this section is dedicated to the analysis of synchrony, and the equilibrium at the origin. First we notice that complete synchrony, $\mathcal{S}:=\{ (q,p) \ | \ q_1 = \dots =q_N \ \text{and} \ p_1 = \dots = p_N  \}$, is always an invariant space for the system \eqref{eq:equations_motion}, such a claim can be easily checked by direct inspection of \eqref{eq:equations_motion}. In fact, the fully synchronised system behaves as a simple pendulum; implying that bifurcations of synchronous equilibria can only happen transversely to the synchrony space. The following proposition shows that the eigenvalues associated to synchrony can only be purely real or purely imaginary, forbidding Hamiltonian-Hopf bifurcations \cite{van2006hamiltonian, van1990hamiltonian, chossat2002hamiltonian}, where for some values of the parameter the eigenvalues take complex values, i.e, $\Re(\lambda)\neq0$, $\Im(\lambda)\neq0$.  

\begin{Proposition}\label{prop:synch_eig}
    The eigenvalues associated to synchronous states $(q^\text{synch},p^\text{synch}) \in \mathcal{S}$ are
    \begin{equation}\label{eq:eig}
      \pm  \sqrt{-     \left(  1 + 2 \kappa c_{01} \lambda  \right) \left( \cos{q^\text{synch}}  + 2 \kappa c_{10} \lambda \right)} ,
    \end{equation}
    where $\lambda \in \text{Spec}(L)$, and $L$ is the corresponding graph Laplacian.
\end{Proposition}

\begin{proof}
    We start by computing the Jacobian of the vector field \eqref{eq:equations_motion},
    \begin{equation}
        J=\begin{pmatrix}
            \frac{\partial\dot{q}}{\partial q} &  \frac{\partial\dot{q}}{\partial p} \\
             \frac{\partial\dot{p}}{\partial q} &  \frac{\partial\dot{p}}{\partial p} 
        \end{pmatrix} .
    \end{equation}
Evaluating the various components we get
\begin{align}
      \frac{\partial\dot{q}}{\partial q} &=  \kappa L^{(11)} ,\\ 
      \frac{\partial\dot{q}}{\partial p} &= \mathbb{I} + \kappa L^{(02)} , \\
      \frac{\partial\dot{p}}{\partial q} &= - \text{diag}\left(\text{Cos}(q)\right) - \kappa L^{(20)} ,\\
      \frac{\partial\dot{p}}{\partial p} &= - \kappa L^{(11)} ,
\end{align}
where $\mathbb{I}$ is the $N$-dimensional identity matrix, $\text{Cos}(q) = (\cos{q_1}, \dots, \cos{q_N})^\intercal$, and the tensors $L^{(11)}$, $ L^{(20)}$, $ L^{(02)}$ are weighted Laplacians of the graph $\mathcal{G}$, 
\begin{equation}
    L_{ij}^{(nm)}=
    \begin{cases}
    -w_{ij}^{(nm)}, & i\neq j\\
    \sum_{j=1 , j\neq i}^N w_{ij}^{(nm)}, & j=i
    \end{cases}
\end{equation}
where $n,m \in \mathbb{N}$ such that $n+m=2$, and the weights depend on position and momenta as follows
\begin{align}
    w_{ij}^{(nm)} &:= G_{nm}(q_i-q_j,p_i-p_j) .
\end{align}
Notice that, thanks to the form of the interaction $G$, the weights are symmetric. Now, we restrict the Jacobian to the synchrony space, which leads to the weights
\begin{align}
    w_{ij}^{(11)} &:= G_{11}(0,0) = 0 , \\ 
    w_{ij}^{(20)} &:= G_{20}(0,0) = 2 c_{10} , \\
    w_{ij}^{(02)} &:= G_{02}(0,0) = 2 c_{01} .
\end{align}
So the Jacobian can be written as 
    \begin{equation}
        J=\begin{pmatrix}
            0 &  \mathbb{I} + 2 \kappa c_{01} L \\
             - \cos{q^\text{synch}} \mathbb{I} - 2 \kappa c_{10} L &  0 
        \end{pmatrix} ,
    \end{equation}
where $L$ is the graph Laplacian and $q^\text{synch}:=q_1 = \dots =q_N$. In order to compute the eigenvalues we need to solve the equation
 \begin{equation}
        \det \begin{pmatrix}
            - \lambda \mathbb{I} &  \mathbb{I} + 2 \kappa c_{01} L \\
             - \cos{q^\text{synch}} \mathbb{I} - 2 \kappa c_{10} L &  - \lambda \mathbb{I} 
        \end{pmatrix} = 0 ,
    \end{equation}
which, thanks to the properties of block matrices \cite{silvester_2000}, can be reduced to
\begin{equation}
    \det\left(\lambda^2 \mathbb{I} - \left(  \mathbb{I} + 2 \kappa c_{01} L  \right) \right(- \cos{q^\text{synch}} \mathbb{I} - 2 \kappa c_{10} L\left) \right) =0 .
\end{equation}
So, the eigenvalues are plus/minus the square roots of the eigenvalues of the matrix
\begin{equation}\label{eq:mat}
-     \left(  \mathbb{I} + 2 \kappa c_{01} L  \right) \left( \cos{q^\text{synch}} \mathbb{I} + 2 \kappa c_{10} L\right) .
\end{equation}
Notice that every term of the matrix \eqref{eq:mat} is a symmetric matrix, therefore the eigenvalues of \eqref{eq:mat} are real. As a consequence the eigenvalues of $J$ are square roots of real numbers that can either lead to real or imaginary numbers. Let $S$ be the similarity matrix taking $L$ to the Jordan form $D$, i.e., $S^{-1}LS=D=\text{diag}\left((\lambda_1, \dots, \lambda_N)^\intercal\right)$, where the strictly diagonal structure is guaranteed since $L$ is symmetric. We notice that $S$ diagonalises \eqref{eq:mat},
\begin{align}
&-   S^{-1}  \left(  \mathbb{I} + 2 \kappa c_{01} L  \right) \left( \cos{q^\text{synch}} \mathbb{I} + 2 \kappa c_{10} L\right) S , \\
&-   S^{-1}  \left(  \mathbb{I} + 2 \kappa c_{01} L  \right) S S^{-1}\left( \cos{q^\text{synch}} \mathbb{I} + 2 \kappa c_{10} L\right) S , \\
&-    \left(  \mathbb{I} + 2 \kappa c_{01} D  \right) \left( \cos{q^\text{synch}} \mathbb{I} + 2 \kappa c_{10} D \right) .
\end{align}
Now \eqref{eq:mat} is in diagonal form and the statement follows.
\end{proof}

\begin{Remark}
    For connected graphs, $0$ is a simple eigenvalue of $L$, so we can also rewrite the expression \eqref{eq:eig} as
     \begin{equation}\label{eq:eig_simp}
      \left\{ \pm  \sqrt{-     \cos{q^\text{synch}}}, \pm  \sqrt{-     \left(  1 + 2 \kappa c_{01} \lambda_l  \right) \left( \cos{q^\text{synch}}  + 2 \kappa c_{10} \lambda_l \right)} , \ l=2,\dots, N \right\} .
    \end{equation}
    For the equilibrium at the origin the eigenvalues reduce to 
    \begin{equation}\label{eq:eig_or}
      \left\{ \pm  i , \pm  \sqrt{-     \left(  1 + 2 \kappa c_{01} \lambda_l  \right) \left( 1  + 2 \kappa c_{10} \lambda_l \right)} , \ l=2,\dots, N \right\} .
    \end{equation}
\end{Remark}

We recall that the equilibria of the uncoupled system can be classified in three main classes depending on the combinations of the individual equilibria of the pendula, \emph{elliptic}: combinations of centers, \emph{saddle}: combinations of saddles, and \emph{mixed}: combinations of centers and saddles. Since the coupling does not produce new classes of equilibria, we use the same nomenclature for the equilibria of the coupled system.

We already mentioned that the synchronised system behaves as a simple pendulum. As a consequence there cannot be bifurcations happening in the synchrony space. Such a property can also be verified by considering the eigenvalues \eqref{eq:eig_simp}, indeed the first eigenvalue does not depend on $\kappa$ and the eigenvector associated to it is the vector with all components equal to one, i.e., $\mathbf{1}=(1, \dots,1)^\intercal$. Actually, it is possible to know the eigenvector of each eigenvalue, as they are the same of the graph Laplacian, since the matrix diagonalising $L$ also diagonalises the Jacobian of the synchronised system. It is quite remarkable that some of the eigenvectors of the Laplacian actually belong to anti-synchrony spaces defined by odd-balanced partitions. In particular for the origin it means that the bifurcation takes place in an anti-synchrony invariant space.

\begin{Lemma}\label{lm:bif_origin}
    Let $c_{10/01} < 0$, then the origin undergoes a bifurcation elliptic to mixed, or vice versa, for
    \begin{equation}\label{eq:crit_origin}
       \kappa=\kappa^\textup{crit} := - \frac{1}{2c_{10/01} \lambda_l} ,
    \end{equation}
    $l \in \{2, \dots, N\}$. The bifurcation takes place in a centre manifold \cite{carr2012applications, vanderbauwhede1989centre} that at the origin is tangent to the vector $({v_l}_1, \dots, {v_l}_N  ,{v_l}_1, \dots, {v_l}_N)^\intercal$, where $v_l$ is the eigenvector of $L$ with eigenvalue $\lambda_l$.
\end{Lemma}


Let us consider bivalent and trivalent graphs, i.e., graphs having at least one eigenvector with components in $\mathcal{B}=\{-1,1\}$ or $\mathcal{T}=\{-1,0,1\}$ respectively; we call such eigenvectors $v_\mathcal{B}$ and $v_\mathcal{T}$. For a classification of bivalent and trivalent graphs and their properties see \cite{caputo2019graph, caputo2023eigenvectors, alencar2023graphs}.


\begin{Proposition}\label{prop:odd_iff}
    A graph admits an odd-balanced partition if and only if it is bivalent or trivalent. 
\end{Proposition}

\begin{proof}
    We start by assuming that the graph is bivalent or trivalent. Then, by the definition of $v_{\mathcal{B}/\mathcal{T}}$, the odd-balanced relations we are considering are the ones with only one letter and, possibly, zeros. The pattern is identified by assigning to the node $i$ the $i$-th component of $v_{\mathcal{B}/\mathcal{T}}$. We can check that the eigenspace associated to $v_{\mathcal{B}/\mathcal{T}}$ is invariant. Let us consider the odd-difference coupled vector field \cite{neuberger2019synchrony},
    \begin{equation}\label{eq:odd_difference_vf}
        f_i(x)=g(x_i)+ \sum_{j=1}^N a_{ij} h(x_i - x_j) ,
    \end{equation}
    where $g$ and $h$ are analytic odd functions. We split the set of nodes in three sub-sets $\mathcal{P}=\{ i\in \mathcal{V} \ | \ v_{\mathcal{B}/\mathcal{T}}^i =1 \}$, $\mathcal{N}=\{ i\in \mathcal{V} \ | \ v_{\mathcal{B}/\mathcal{T}}^i =-1 \}$, $\mathcal{Z}=\{ i\in \mathcal{V} \ | \ v_{\mathcal{B}/\mathcal{T}}^i =0 \}$. One can notice that the sub-sets of nodes we introduced can be associated to an odd-balanced partition. In particular, the nodes in $\mathcal{P}$ and $\mathcal{N}$ are in relative anti-synchrony. Restricting the vector field along $v_{\mathcal{B}/\mathcal{T}}$, we obtain
    \begin{align}
        f_{i \in \mathcal{P}}(s \cdot v_{\mathcal{B}/\mathcal{T}})&=g(s)+ d_\mathcal{N} h(2s) + d_\mathcal{Z} h(s) , \\
        f_{i \in \mathcal{N}}(s \cdot  v_{\mathcal{B}/\mathcal{T}})&=-g(s)- d_\mathcal{P} h(2s) - d_\mathcal{Z} h(s) , \\
        f_{i \in \mathcal{Z}}(s \cdot v_{\mathcal{B}/\mathcal{T}})&=0, 
    \end{align}
    where $s \in \mathbb{R}$, $d_{\mathcal{N}/\mathcal{P}}:= \sum_{j \in \mathcal{N}/\mathcal{P}} a_{ij}$ is the number of negative/positive nodes connected to a positive/negative node, and similarly $d_\mathcal{Z}:= \sum_{j \in \mathcal{Z}} a_{ij}$ is the number of zero nodes connected to a non-zero node. As a consequence of the results in \cite{caputo2019graph}, we have that the quantities just defined do not depend on the specific node but just on the classes $\mathcal{P}$, $\mathcal{N}$, $\mathcal{Z}$ and therefore the space induced by $v_{\mathcal{B}/\mathcal{T}}$ is invariant. So, thanks to Theorem \ref{thm:odd}, we have that $v_{\mathcal{B}/\mathcal{T}}$ defines an odd balanced partition.

    Now we prove the other direction. We start with a graph admitting an odd-balanced partition. It is straightforward to check that also for the vector field \eqref{eq:odd_difference_vf}, similarly to what we already have seen for the coupled pendula we are studying, the Jacobian at the origin has the same eigenvectors as the Laplacian matrix. Moreover, the invariant space $\mathcal{A}$ associated with the odd-balanced relation contains the origin. So, there is at least one eigenvector of $L$ that lies in $\mathcal{A}$. Since $\mathcal{A}$ is a linear invariant space one can check that for analytic nonlinear odd functions $g$ only sign relations act linearly on it (we conclude the proof by showing such result), leaving as the only option the vectors $v_{\mathcal{B}/\mathcal{T}}$. 
    
    Let $g(x)$ be an analytic nonlinear odd function
    \begin{equation}
        g(x)=\sum_{n=0}^\infty c_n x^{2n+1} ,
    \end{equation}
    where $c_n \in \mathbb{R}$, $\forall n$. We consider a linear transformation of $x$, i.e., $x \mapsto \alpha x + \beta$, where $\alpha, \beta \in  \mathbb{R}$. So, we look for the linear transformations that act linearly on the nonlinear function $g(x)$, which means
    \begin{equation}
        \alpha \sum_{n=0}^\infty c_n x^{2n+1} + \beta = \sum_{n=0}^\infty c_n (\alpha x + \beta)^{2n+1}.
    \end{equation}
    Since there is no assumption on the coefficients, we consider the equation 
    \begin{equation}
         \alpha x^{2n+1} + \beta =  (\alpha x + \beta)^{2n+1} ,
    \end{equation}
    $\forall n \in \mathbb{N}$. Using the Binomial Theorem, and isolating the first and last term we obtain
    \begin{equation}
         \alpha x^{2n+1} + \beta =  \alpha^{2n + 1} x^{2n + 1} + \beta + \sum_{k=1}^{2n} \binom{2n+1}{k}  (\alpha x)^{2n + 1-k} \beta^k    .
    \end{equation}
    We notice that the first condition needed is $\alpha^{2n}=1$, that leads to the real solutions $\alpha=\pm 1$. We are left with the equation
    \begin{equation}
           \sum_{k=1}^{2n} \binom{2n+1}{k}  (\pm x)^{2n + 1-k} \beta^k  = 0   ,
    \end{equation}
    from which we obtain $\beta=0$.
\end{proof}

Let us restrict the study of bifurcations to the anti-synchrony subspaces induced by the vectors $v_{\mathcal{B}/\mathcal{T}}$. The computations performed in the proof of Proposition \ref{prop:odd_iff} can be done for the system of coupled pendula \eqref{eq:equations_motion}, obtaining the vector field
\begin{equation}
\begin{aligned}
     \dot{q}_\mathcal{P} &= p_\mathcal{P} + \kappa d_\mathcal{N} G_{01}(2q_\mathcal{P},2p_\mathcal{P}) + \kappa d_\mathcal{Z} G_{01}(q_\mathcal{P},p_\mathcal{P}) , \\
     \dot{p}_\mathcal{P} &= - \sin{q_\mathcal{P}} - \kappa d_\mathcal{N} G_{10}(2q_\mathcal{P},2p_\mathcal{P}) - \kappa d_\mathcal{Z} G_{10}(q_\mathcal{P},p_\mathcal{P}) , \\
     \dot{q}_\mathcal{N} &= - p_\mathcal{N} - \kappa d_\mathcal{P} G_{01}(2q_\mathcal{N},2p_\mathcal{N}) - \kappa d_\mathcal{Z} G_{01}(q_\mathcal{N},p_\mathcal{N}) , \\
     \dot{p}_\mathcal{N} &=  \sin{q_\mathcal{N}} + \kappa d_\mathcal{P} G_{10}(2q_\mathcal{N},2p_\mathcal{N}) + \kappa d_\mathcal{Z} G_{10}(q_\mathcal{N},p_\mathcal{N}) , \\
     \dot{q}_\mathcal{Z} &= 0 , \\
      \dot{p}_\mathcal{Z} &= 0 ,
\end{aligned}    
\end{equation}
effectively reducing the study to a $1$-d.o.f. system
\begin{equation}\label{eq:reduced}
\begin{aligned}
     \dot{x} &= y + \kappa d_\pm G_{01}(2x,2y) + \kappa d_0 G_{01}(x,y) , \\
     \dot{y} &= - \sin{x} - \kappa d_\pm G_{10}(2 x,2 y) - \kappa d_0 G_{10}(x , y) ,
\end{aligned}    
\end{equation}
where $d_\pm= d_\mathcal{P} = d_\mathcal{N}$ and $d_0 = d_\mathcal{Z}$. 

\begin{Lemma}
    Let $v_{\mathcal{B}/\mathcal{T}}$ be a bivalent/trivalent eigenvector of $L$ with eigenvalue $\lambda$, then the following equality holds
    \begin{equation}
        \lambda = 2 d_\pm + d_0 . 
    \end{equation}
\end{Lemma}

\begin{proof}
    We compute the Jacobian of \eqref{eq:reduced} at the origin, and from the diagonalisation it follows that the eigenvalues are 
    \begin{equation}\label{eq:eig_second}
        \pm \sqrt{-\left( 1+ 2 \kappa c_{10} (2 d_\pm + d_0) \right)\left( 1+ 2 \kappa c_{01} (2 d_\pm + d_0) \right) } .
    \end{equation}
Comparing \eqref{eq:eig_second} with \eqref{eq:eig_or} the statement follows.
\end{proof}

Notice that, we can write down an Hamiltonian inducing the vector field \eqref{eq:reduced}, i.e.,
\begin{equation}\label{eq:h_red}
    K = \frac{y^2}{2} -\cos{x} + \kappa \frac{d_\pm}{2}  G(2 x,2 y) + \kappa d_0 G(x , y) .
\end{equation}
By expanding $K$ near the origin up to fourth order, and reordering the terms, we obtain
%
\begin{equation}\label{eq:h_red_exp}
    \begin{aligned}
    K -k_0 &= x^4 \left(\kappa c_{20}  (8 d_\pm +d_0) - \frac{1}{24} \right) + y^4 \kappa c_{02}  (8 d_\pm +d_0) + x^2 y^2 \kappa c_{11} \left( \frac{4}{3} d_\pm +\frac{1}{6} d_0 \right) + \\
    & + y^2 \left( \frac{1}{2} + \kappa \lambda c_{01} \right) + x^2 \left( \frac{1}{2} + \kappa \lambda c_{10} \right)  +\text{h.o.t.} ,
\end{aligned}
\end{equation}
where $k_0$ is a constant. We notice that the coefficients of $p^2$ and $q^2$ go to zero when $\kappa$ reaches the critical values for bifurcation, i.e., $\kappa=1/(2 \lambda c_{01})$ or $\kappa=1/(2 \lambda c_{10})$. Therefore, expression \eqref{eq:h_red_exp} resembles the unfolding, under the symmetry conditions proper of the system, of the $X_9$ singularity \cite{Arnold2012SingularitiesOD}, 
\begin{equation}\label{eq:x9}
    x^4 +y^4 +a x^2y^2 ,
\end{equation}
where $a^2\neq4$, also known as double cusp \cite{upstill1982double}.

\begin{Proposition}\label{prop:pitchfork}
    At the critical values $\kappa^\textup{crit}\in\left\{\frac{1}{2 c_{10} \lambda},\frac{1}{2 c_{01} \lambda} \right\}$, where $\lambda$ is an eigenvalue of the Laplacian associated to a bivalent/trivalent eigenvector, the equilibrium at the origin of \eqref{eq:reduced} undergoes pitchfork bifurcations.
\end{Proposition}

\begin{proof}
    Let us write down again the equations we are considering,
    \begin{equation}
\begin{aligned}\label{eq:eq_invariant_reduced}
     \dot{x} &= y + \kappa d_\pm G_{01}(2x,2y) + \kappa d_0 G_{01}(x,y) =: f(x,y,\kappa) , \\
     \dot{y} &= - \sin{x} - \kappa d_\pm G_{10}(2 x,2 y) - \kappa d_0 G_{10}(x , y) =: g(x,y,\kappa) .
\end{aligned}    
\end{equation}
We notice that $G_{01}$ does not contain the coefficient $c_{10}$ and similarly $G_{10}$ does not contain the coefficient $c_{01}$. So, generically, i.e., for $c_{10}\neq c_{01}$, the two bifurcations occur separately, reducing the analysis to two codimension one bifurcations, respectively of $f(x,y,\kappa)$ and $g(x,y,\kappa)$.The splitting of the bifurcations can be understood more explicitly by noticing that the following equations hold
\begin{align}
    &G_{10}(0,y)=0 , \\&G_{01}(0,y)=\sum_{m=1}^\infty c_{0m} 2m y^{2m-1} ,\\
    &G_{10}(x,0)=\sum_{l=1}^\infty c_{l0} 2l x^{2l-1} , \\  &G_{01}(x,0)=0 .
\end{align}
So if we restrict \eqref{eq:eq_invariant_reduced} to $x=0$ we have
\begin{equation}
\begin{aligned}
     \dot{x} &=  f(0,y,\kappa) , \\
     \dot{y} &= 0 ,
\end{aligned}
\end{equation}
while restricting \eqref{eq:eq_invariant_reduced} to $y=0$ we obtain 
\begin{equation}
\begin{aligned}
     \dot{x} &=  0 , \\
     \dot{y} &= g(x,0,\kappa) .
\end{aligned}
\end{equation}
Let us recall that, given a function $h(x,y,\kappa)$, in order to have a pitchfork bifurcation of the origin  along the $x$ coordinates at $\kappa^\text{crit}$ we need that $h$ is an odd function, and
    \begin{align}
        h(0, 0,\kappa^\text{crit}) = 0 , \quad
        &\frac{\partial h}{\partial x} (0, 0,\kappa^\text{crit}) = 0 , \\
        \frac{\partial^2 h}{\partial x^2} (0, 0,\kappa^\text{crit}) = 0 , \quad
        &\frac{\partial^3 h}{\partial x^3} (0, 0, \kappa^\text{crit}) \neq 0 , \\
        \frac{\partial h}{\partial \kappa} (0, 0, \kappa^\text{crit}) = 0 , \quad
        &\frac{\partial^2 h}{\partial \kappa \partial x} (0, 0, \kappa^\text{crit}) \neq 0 .
    \end{align}
Clearly, the same applies if we look along the $y$ direction. Such conditions are verified for $f(x,y, \kappa)$ along $y$ and $g(x,y, \kappa)$ along $x$. We report the non-degeneracy conditions: for $f$ with $\kappa^\text{crit}=\frac{1}{2 c_{01} \lambda}$ we have
\begin{align}
    \frac{\partial^3 f}{\partial y^3} (0, 0,\kappa^\text{crit}) &= - \frac{12 c_{02}}{\lambda c_{01}} \left( 8 d_\pm +d_0 \right) ,\\
     \frac{\partial^2 f}{\partial \kappa \partial y} (0, 0,\kappa^\text{crit}) &= 2 \lambda c_{01} ,
\end{align}
while for $g$ with $\kappa^\text{crit}=\frac{1}{2 c_{10} \lambda}$ we have
\begin{align}
    \frac{\partial^3 g}{\partial x^3} (0, 0,\kappa^\text{crit}) &= 1 + \frac{12 c_{20}}{\lambda c_{10}} \left( 8 d_\pm +d_0 \right) ,\\
     \frac{\partial^2 g}{\partial \kappa \partial x} (0, 0,\kappa^\text{crit}) &= 2 \lambda c_{10} .
\end{align}
\end{proof}

\begin{Remark}
    The non-degeneracy conditions consistently match with the non-degeneracy requirements that one would expect from the expression of the singularity \eqref{eq:h_red_exp}. Moreover, from equation \eqref{eq:h_red_exp}, comparing with the normal form \eqref{eq:x9} for the singularity $X_9$, we can read a further non-degeneracy condition, $a^2 \neq 4$ that involves the structure of the singularity. 
\end{Remark}

At this point, one might be interested to understand nonlocal properties of the system after the bifurcations, such as the existence of \emph{homoclinic} orbits. Let us recall that the equations in the invariant spaces defined by $v_{\mathcal{B}/\mathcal{T}}$ can be associated to a Hamiltonian function, see \eqref{eq:h_red} and \eqref{eq:h_red_exp}, which is associated to the double cusp singularity \eqref{eq:x9}. We consider a function that mimics the structure of \eqref{eq:h_red_exp} as an unfolding of \eqref{eq:x9},
\begin{equation}\label{eq:mimics}
   F= x^4 + y^4 + a x^2 y^2 +\alpha x^2 +\beta y^2 ,
\end{equation}
where $\alpha, \beta \in \mathbb{R}$, and $a$ is the same as in \eqref{eq:x9}. The bifurcations of \eqref{eq:mimics} take place at $\alpha=0$ and $\beta=0$. We notice that the level-set of the origin is given by $F=0$, so for $\alpha$ and $\beta$ negative only the origin belongs to the level-set. Analogously to what happens for the cusp singularity \cite{efstathiou2012topology}, when either $\alpha$ or $\beta$ cross zero a homoclinic orbit appear. However, when both bifurcations have occurred, we see the appearance of \emph{heteroclinic} connections. In Figure \ref{fig:homoclinic}, we show some representative examples of homoclinic and heteroclinic solutions appearing as level-sets of \eqref{eq:mimics}.

\begin{figure}[ht]
     \centering
     \begin{subfigure}[b]{0.3\textwidth}
         \centering
         \includegraphics[width=\textwidth]{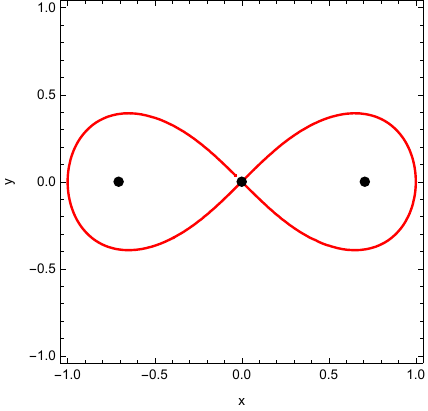}
         \caption{ $\alpha=-1$, $\beta=1$}
         \label{fig:bif_x}
     \end{subfigure}
     \begin{subfigure}[b]{0.3\textwidth}
         \centering
         \includegraphics[width=\textwidth]{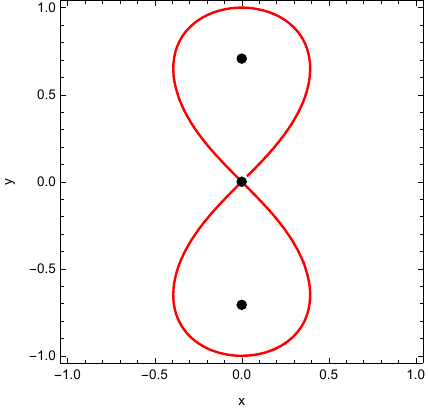}
         \caption{$\alpha=1$, $\beta=-1$}
         \label{fig:bif_y}
     \end{subfigure}
     \begin{subfigure}[b]{0.3\textwidth}
         \centering
         \includegraphics[width=\textwidth]{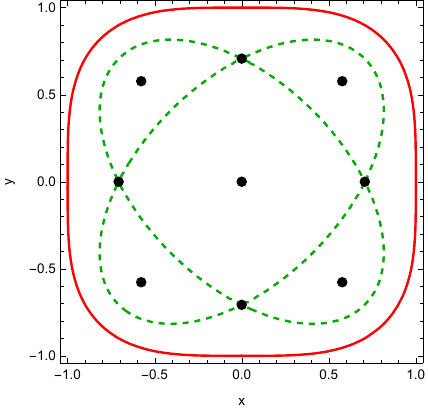}
         \caption{$\alpha=-1$, $\beta=-1$}
         \label{fig:bif_xy}
     \end{subfigure}
        \caption{Homoclinic and heteroclinic solutions as level-sets of \eqref{eq:mimics}. We have set $a=1$ as it does not affect the qualitative structure of the sets. The black dots are the critical points of \eqref{eq:mimics}, while the red curves are given by $F=0$, i.e., they belong to the same level-set of the origin. Figure \ref{fig:bif_x} and  \ref{fig:bif_y} represent the level-sets associated to the critical points respectively after a bifurcation along the $x$ and $y$ axis. Figure \ref{fig:bif_xy} represents the situation after both bifurcations have occurred, here the green-dashed orbits belong to the level-set $F=-1/4$.}
        \label{fig:homoclinic}
\end{figure}


\section{Large Graphs}

Studying networks with a large number of nodes is relevant for understanding the statistical properties of large, but finite, interacting systems. In particular, in the Hamiltonian context, there are relevant connections between ergodic theory and KAM theory. As a matter of fact, systems with large numbers of particles have been successfully studied utilizing statistical mechanics, where one assumes the ergodic hypothesis. In other words, to successfully apply the statistical techniques one has to assume that the systems after a sufficiently long time will visit almost all the points in the phase space. So a natural question arises, when is it possible to assume the ergodic hypothesis, or equivalently, when and where is the system chaotic? Now, KAM theory asserts that for sufficiently small perturbations most of the orbits will show quasi-periodic motion. This seems, at first glance, to forbid ergodicity for weakly coupled systems. If we consider the problem with more care, we realize that in such a context quantifying the meaning of \emph{small} perturbation becomes crucial.

We dedicate this section to studying the behaviour of the system under investigation for networks with a large number of nodes, $N \gg 1$. Specifically, we identify some relations between the graph properties and the range in which the bifurcations of the systems occur. In Section \ref{sec:simulations}, we show how the bifurcations of the origin are associated with the emergence of weak chaos in regions of low energy. Combining the results we can provide a step connecting the regimes of ergodicity with the network structure and the \enquote{smallness} of the interaction.

\begin{Proposition}\label{prop:intervals}
    The critical values \eqref{eq:crit_origin} associated to bifurcations of the origin are contained in the intervals
    \begin{equation}\label{eq:interval}
        -\frac{1}{2 c_{10}} \cdot \left[ \frac{1}{N}, \frac{N}{2 \kappa'} \right] \ \text{and} \  -\frac{1}{2 c_{01}} \cdot \left[ \frac{1}{N}, \frac{N}{2 \kappa'} \right],
    \end{equation}
    where $\kappa'$ is the edge-connectivity of the graph, i.e., the minimum number of edges whose deletion would lead to a disconnected graph.
\end{Proposition}


\begin{proof}
    We recall that the critical values for bifurcations of the origin are $\kappa^\text{crit} = - 1/(2c_{10/01} \lambda_l)$, which immediately explains the scalar factor in front of the interval \eqref{eq:interval}. We notice that the interval where the critical values are contained is given by the largest and the smallest non-zero eigenvalues of the Laplacian, the latter is also known as \emph{spectral gap} or \emph{algebraic connectivity}. The largest eigenvalue of $L$, denoted by $\lambda_N$ satisfies the inequality $\lambda_N \leq N$, see \cite{cvetkovic2009introduction}. The inverse relation gives the left extremum of the interval \eqref{eq:interval}. For the right extremum we use a Cheeger-type inequality found by Oshikiri \cite{oshikiri2002cheeger, cvetkovic2009introduction}, $\lambda_2 \geq 2 \kappa' / N$.
\end{proof}

\begin{Lemma}
    Let $\mathcal{G}_{i_1}, \mathcal{G}_{i_2}, \dots$ be an infinite sequence of graphs, where the labels $i_1 < i_2 < \dots $ are the number of nodes of the respective graph. If the edge-connectivity is proportional to the number of nodes of the graph, i.e., $\kappa' \propto i_n$, then the interval \eqref{eq:interval} in which bifurcations of the origin occur converge to a finite interval in the limit of the sequence $\mathcal{G}_{i_n \to \infty}$.
\end{Lemma}

A sequence of graphs satisfying the property $\kappa' \propto N$ can be constructed by following the lines of the next theorem.

\begin{Theorem}[\cite{behzad1971introduction}]
    A graph $\mathcal{G}=\{ \mathcal{V}, \mathcal{E}\}$ has edge-connectivity $\kappa'= M$ if and only if there exist no nonempty proper subset $\mathcal{W}$ of $\mathcal{V}$ such that the number of edges joining $\mathcal{W}$ and $\mathcal{V} \setminus \mathcal{W}$ is less than $M$.
\end{Theorem}

We look at two relevant examples that show the extreme cases, in terms of edge-connectivity, for sequences of graphs with increasing number of nodes. The firsts are the path graphs, $P_N$, $N \in \mathbb{N}$, which always have $\kappa'= 1$. We expect that the interval in which the bifurcations occur becomes unbounded in the limit $N \to \infty$. Actually, we can compute explicitly the spectrum of the Laplacian for $P_N$, obtaining $\lambda_j = 2 - 2 \cos{(j \pi/N)}$. So in the limit for \eqref{eq:interval} we get
\begin{equation}
        \left[ -\frac{1}{8 c_{10/01}}, \infty \right[ .
\end{equation}
So, for path graphs with a large number of nodes the bifurcations are spread in a large range of $\kappa$, starting from a value close to $-1/(8 c_{10/01})$. As a consequence, the weak coupling regime at low energies lead to regular perturbations of the former system.
On the other hand, as a second example, we consider the complete graphs $K_N$, $N \in \mathbb{N}$, which have $\kappa'= N-1$, and therefore we expect a compact interval in the limit. Complete graphs are very degenerate and the Laplacian spectrum has one zero eigenvalue and all the other eigenvalues are equal to $N$. Then, all bifurcations occur at the same point $-1/(2 c_{10/01} N)$, that in the limit converge to zero. Clearly, the situation for complete graphs is dramatically different with respect to the path graphs. The bifurcations occurs all simultaneously at a value of $\kappa$ that is inversely proportional to the number of nodes. So for large complete graphs the the weak coupling can lead to complex dynamics and chaos at low energies.

\section{Simulations}\label{sec:simulations}



In this section, we use numerical methods to investigate the interplay between the bifurcation of equilibria and the emergence of chaos. We showcase through some examples how the bifurcation of the origin is associated with transitions of nearby trajectories from regular to chaotic. It is known that for Hamiltonian systems the transition to chaos is associated with the destruction of KAM tori, which in turn depends on the energy and coupling parameter \cite{ott2002chaos}. For two coupled pendula with harmonic interaction, it has been shown that the arising of chaos is indeed related to the total energy of the system, with an onset energy considerably distant from the minimum \cite{huynh2010two}. Reasonably, such an effect is related to an increasing prevalence of nonlinearities as the energy grows. Interestingly, the transitions we observe take place at \emph{low energy}, i.e., close to the minimum of the Hamiltonian. Arguably the geometry induced by the bifurcation of the origin is a key feature for understanding the mechanism for the appearance of chaos.   

For the following computations, we consider as interaction function a double-well potential
\begin{equation}
    G(x,y) = 1/4 -x^2 + x^4 ,
\end{equation}
inducing pitchfork bifurcations of the origin at the critical value $\kappa^\text{crit}=1/(2\lambda)$, where $\lambda$ is a non-zero eigenvalue of the Laplacian. We compute the Lyapunov exponents, by implementing the algorithm prescribed in \cite{sandri1996numerical}, to estimate the chaoticity of the simulated orbits, see Table \ref{tab:lce}.

\begin{table}[h!]
\begin{center}
\begin{tabular}{||c | c | c||} 
 \hline
 Network & Energy  & Lyapunov Exponents \\ [0.5ex] 
 \hline\hline
 \multirow{2}{2em}{$K_2$} & $E\approx-1.92$ &    $\{1.4 \cdot 10^{-3}, 1.5\cdot 10^{-3},  -1.4 \cdot 10^{-3}, -1.5\cdot 10^{-3}\}$  \\ 
  & $E\approx-1.85$&   $\{8.4\cdot 10^{-2}, 1.9\cdot 10^{-3}, -1.6\cdot 10^{-3}, -8.4 \cdot 10^{-2}\}$  \\ 
 \hline
 \multirow{2}{2em}{$K_3$} & $E\approx-2.87$&   $\{ 1.8\cdot 10^{-3}, 7.9\cdot 10^{-4},  7.0\cdot 10^{-4}, - 8.7 \cdot 10^{-4}, -5.2 \cdot 10^{-4}, -1.9  \cdot 10^{-3}\}$  \\ 
  & $E\approx-2.78$ &  $\{7.4 \cdot 10^{-2}, 4.4 \cdot 10^{-2}, 8.3 \cdot 10^{-4}, -1.4 \cdot 10^{-3}, -4.2 \cdot 10^{-2}, - 7.6 \cdot 10^{-2}\}$  \\
 \hline
 \multirow{3}{2em}{$P_3$} & $E\approx-2.90$ &   $\{1.8 \cdot 10^{-3}, 1.6 \cdot 10^{-3}, 3.2 \cdot 10^{-4}, -9.5 \cdot 10^{-3}, -2.2 \cdot 10^{-3}, -5.9 \cdot 10^{-4}\}$  \\
  & $E\approx-2.84$ &  $\{ 7.2 \cdot 10^{-2}, -2.5 \cdot 10^{-4}, 5.0 \cdot 10^{-4}, 4.8 \cdot 10^{-4}, -1.4 \cdot 10^{-4}, -7.3 \cdot 10^{-2}\}$    \\
  & $E\approx-2.48$ &  $\{ 3.1 \cdot 10^{-1}, 2.4 \cdot 10^{-2}, 2.9 \cdot 10^{-3}, -2.5 \cdot 10^{-3}, -2.4 \cdot 10^{-2}, -3.2 \cdot 10^{-1}\}$    \\  [1ex] 
 \hline
\end{tabular}
\end{center}
\caption{We report here the values of the Lyapunov exponents, with two digits approximation, in relation with the respective energies, and network structures.  The system has been simulated before and after the occurrence of a bifurcation. Comparing the results, we can notice a jump of one or two order of magnitudes in the positive Lyapunov exponents. The number of the exponents reaching larger magnitudes is related the the number of bifurcations taking place. We also observe that the overall magnitude of the Lyapunov exponents is rather small, with a largest order of magnitude of $10^{-1}$. So, we can infer the system it is weakly chaotic. One can also compare these results with the trajectories of the system. For the complete network with two nodes, $K_2$, with energies $E\approx-1.92$, $E\approx-1.85$ see Figures \ref{fig:trajectory_2_reg}, \ref{fig:trajectory_2_chaos}. For the complete network with three nodes, $K_3$, with energies $E\approx-2.87$, $E\approx-2.78$ see Figures \ref{fig:trajectory_3_comp_reg},  \ref{fig:trajectory_3_comp_chaos}. For the path network with three nodes, $P_3$, with energies $E\approx-2.90$, $E\approx-2.84$, $E\approx-2.48$  see Figures \ref{fig:trajectory_3_path_reg}, \ref{fig:trajectory_3_path_chaos_one},\ref{fig:trajectory_3_path_chaos_two}.}
\label{tab:lce}
\end{table}\textbf{}

\subsection{Two Nodes}

For a coupled system with two oscillators, there are two invariant spaces, at the origin, spanning the phase space: synchrony and anti-synchrony. As already mentioned, in synchrony orbits behave as a simple pendulum for all values of $\kappa$, therefore the bifurcation of the origin takes place in the anti-synchrony space. For such a reason, it is useful to consider what it is happening in the anti-synchrony space, $\mathcal{A}$, and the associated transversal stability. The bifurcation of the origin in $\mathcal{A}$ takes place at $\kappa^\text{crit}=1/4$, while the eigenvalues associated to the transversal eigendirections are $\pm i \sqrt{\cos{x}}$, where $x$ is the position coordinate in $\mathcal{A}$, see Figure \ref{fig:bif_2}.

\begin{figure}[ht]
\centering
\includegraphics[width=0.5\textwidth]{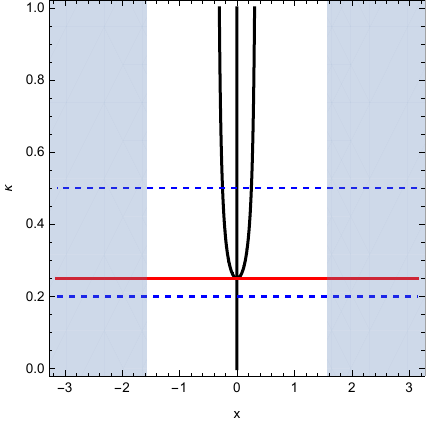}
\caption{The black line represents the pitchfork bifurcation of the origin in the anti-synchrony space $\mathcal{A}$. The shaded regions are transversely unstable with respect to $\mathcal{A}$. We drew a solid red line corresponding to the critical parameter for the bifurcation in $\mathcal{A}$, and two dashed blue lines for the values chosen for the simulations, see Figure \ref{fig:two_nodes}. }
\label{fig:bif_2}
\end{figure}

We performed two simulations of the system for fixed initial conditions close to the origin, one before the bifurcation and one after the bifurcation. It is insightful to see the evolution as a combination of the dynamics inside the anti-synchrony space, and its transversal stability. Before the bifurcation we have a center at the origin of $\mathcal{A}$ and elliptic behaviour orthogonally, therefore we expect regular behaviour, Figure \ref{fig:trajectory_2_reg}. After the bifurcation the origin in $\mathcal{A}$ becomes a saddle, and two centers appear from the pitchfork structure. Simultaneously, a homoclinic orbit will appear, as shown in Figure \ref{fig:bif_x}. The transversal stability is still elliptic. The motion now is a combination of the elliptic behaviour transversal to $\mathcal{A}$ and the unstable wandering induced by the saddle at the origin. We can see in Figure \ref{fig:trajectory_2_chaos} that the outcome is a chaotic trajectory. Let us notice that despite the origin having some unstable eigendirections the motion nearby is still bounded, as predicted by Proposition \ref{prop:bounded_motion}.

\begin{figure}[ht]
     \centering
     \begin{subfigure}[b]{0.45\textwidth}
         \centering
         \includegraphics[width=\textwidth]{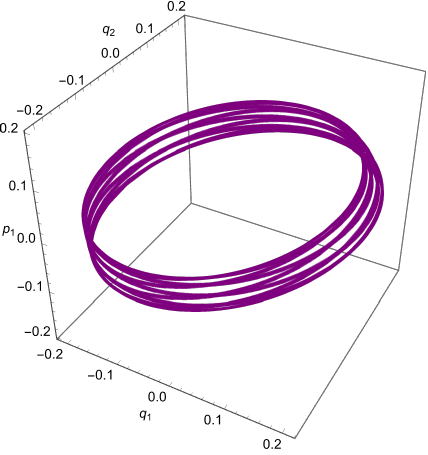}
         \caption{Regular 
         }
         \label{fig:trajectory_2_reg}
     \end{subfigure}
     \begin{subfigure}[b]{0.45\textwidth}
         \centering
         \includegraphics[width=\textwidth]{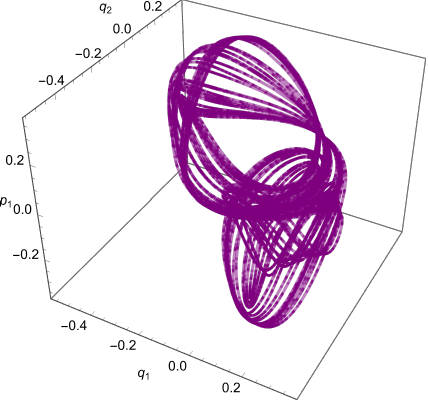}
         \caption{Chaotic 
         }
         \label{fig:trajectory_2_chaos}
     \end{subfigure}
        \caption{The two orbits shown have same initial conditions $(q_1,q_2,p_1,p_2)=(1/5,1/7,0,0)$. For the two d.o.f. system in consideration the orbits shown are proper orbits of the system, as one momentum variable can be fixed by exploiting the conserved energy. The critical parameter for the bifurcation of the origin is given by $\kappa^\text{crit}= 1/4$, the trajectories have respectively $\kappa=1/5$, Figure \ref{fig:trajectory_2_reg}, and $\kappa=1/2$, Figure \ref{fig:trajectory_2_chaos}. After the bifurcation, the orbit transition from regular to chaotic, although the energy is still close to the minimum, see Table \ref{tab:lce}. 
        }
        \label{fig:two_nodes}
\end{figure}

\subsection{Three Nodes: Complete and Path Graph}

As a second example we look at the networks with three nodes, the complete and the path graph. If we compute the linearisation at the origin for the complete graph we find that every eigenvector is associated to a linear invariant space of the system: synchrony, and two equivalent anti-synchrony spaces, $\mathcal{A}^{K_3}$, given by the eigenvectors $(-1,0,1)$ and $(-1,1,0)$. Moreover, we recall that for the complete graph the non-zero eigenvalues of the Laplacian are all equal, so the unique critical value for the bifurcations is $\kappa^\text{crit}=1/6$. Clearly, since we have two equivalent anti-synchrony spaces the bifurcation will occur simultaneously in both of them. If we look at one anti-synchrony space and its transversal behaviour, as we did for the two-node case, we indeed see that the transversal behaviour shows a transition exactly at $\kappa^\text{crit}$, Figure \ref{fig:bif_3_comp}. On the other hand, the Laplacian of the path graph has two different non-zero eigenvalues: $3$ and $1$. However, only the eigenvalue $1$ is associated to an anti-synchrony space, $\mathcal{A}^{P_3}$, defined by the eigenvector $(-1,0,1)$. If we restrict ourselves to $\mathcal{A}^{P_3}$, the origin bifurcates at $\kappa^\text{crit}=1/2$, while the transversal behaviour shows a transition for $\kappa^\text{crit}=1/6$, as expected from the eigenvalues of $L$. In other words, the complete graph represent a degenerate case where the bifurcations appear simultaneously, and the path graph instead is an example of a non-degenerate case. So, for the complete network $K_3$ we simulate two trajectories, one before and one after the bifurcation, Figure \ref{fig:three_nodes_complete}, while for the path network $P_3$ we simulate three orbits one before the first bifurcation, one in between the two bifurcations and one after the second bifurcation, Figure \ref{fig:three_nodes_path}.

\begin{figure}[ht]
     \centering
     \begin{subfigure}[b]{0.45\textwidth}
         \centering
         \includegraphics[width=\textwidth]{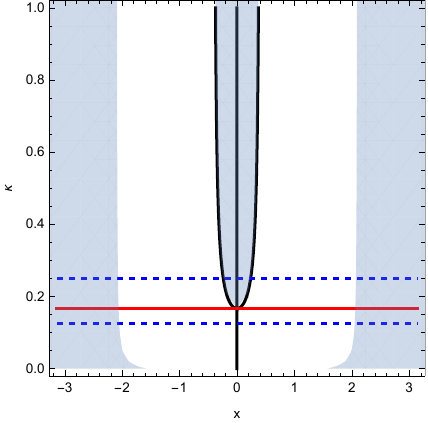}
         \caption{Complete graph $K_3$}
         \label{fig:bif_3_comp}
     \end{subfigure}
     \begin{subfigure}[b]{0.45\textwidth}
         \centering
         \includegraphics[width=\textwidth]{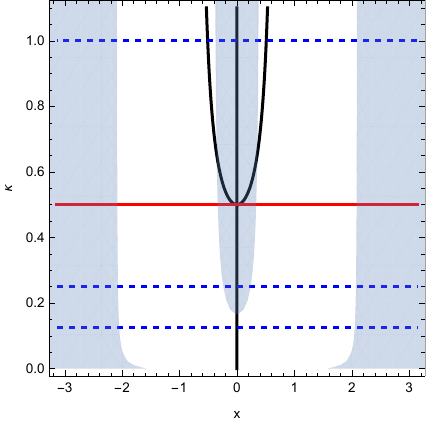}
         \caption{Path graph $P_3$}
         \label{fig:bif_3_path}
     \end{subfigure}
        \caption{The black lines represent the pitchfork bifurcation of the origin in the anti-synchrony spaces respectively $\mathcal{A}^{K_3}$, Figure \ref{fig:bif_3_comp}, and $\mathcal{A}^{P_3}$, Figure \ref{fig:bif_3_path}. The shaded regions are transversely unstable. We drew a solid red line corresponding to the critical parameter in the anti-synchrony space, and dashed blue lines for the values chosen for the simulations, see Figures \ref{fig:three_nodes_complete}, \ref{fig:three_nodes_path}.}
        \label{fig:bif_3}
\end{figure}

Now, considering the restriction coming from the conservation of energy, the phase space is five dimensional. Therefore we choose to show the trajectories in the position variables, resulting in some motion in the three-dimensional space. We notice how the bifurcation of the origin affects the nature of the orbits. When the origin is elliptic the system behaves regularly, Figures \ref{fig:trajectory_3_comp_reg}, \ref{fig:trajectory_3_path_reg}. We then look at the path graph, where the bifurcations take place one at the time, and we see that, after the first bifurcation, the orbit while showing some chaotic behaviour it preserves some regular structure, Figure \ref{fig:trajectory_3_path_chaos_one}. After two bifurcations occurred, the orbit completely loses its regular structure and behaves chaotically, Figures \ref{fig:trajectory_3_comp_chaos}, \ref{fig:trajectory_3_path_chaos_two}. What we have just described is reflected in the Lypunov exponents. Indeed, we notice that after the bifurcations couples of Lyapunov exponents increase noticeably in magnitude, the number of couples is equal to the number of simultaneous bifurcations.

\begin{figure}[ht]
     \centering
     \begin{subfigure}[b]{0.45\textwidth}
         \centering
         \includegraphics[width=\textwidth]{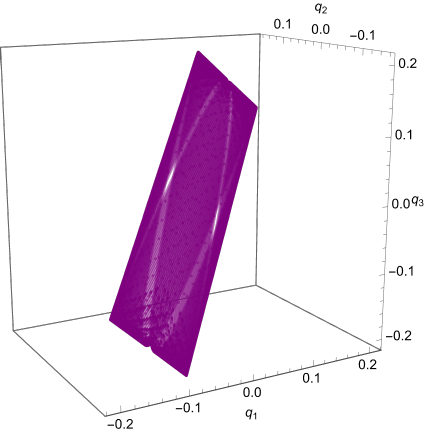}
         \caption{Regular 
         }
         \label{fig:trajectory_3_comp_reg}
     \end{subfigure}
     \begin{subfigure}[b]{0.45\textwidth}
         \centering
         \includegraphics[width=\textwidth]{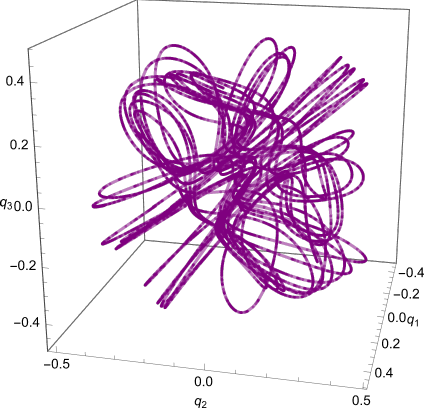}
         \caption{Chaotic 
         }
         \label{fig:trajectory_3_comp_chaos}
     \end{subfigure}
        \caption{The two orbits shown have the same initial conditions $(q_1,q_2,q_3,p_1,p_2,p_3)=(1/5,1/7,1/10,0,0,0)$. The figures represent the position's trajectory of the system. The critical parameter for the bifurcation is $\kappa^\text{crit}=1/6$, the simulations are done for $\kappa=1/8$, Figure \ref{fig:trajectory_3_comp_reg}, and $\kappa=1/4$, Figure \ref{fig:trajectory_3_comp_chaos}. The bifurcation is degenerate as it changes simultaneously the eigenvalues of the origin along two d.o.f.. Such degeneracy is reflected in the Lyapunov exponents, Table \ref{tab:lce}, where we can see that after the bifurcation two positive exponents undergoes a significative change in magnitude towards larger positive values. 
        A small remark regarding Figure \ref{fig:trajectory_3_comp_reg}. The motion of the regular orbit lies on a thick plane, whose thickness depends on the value of $\kappa$. The thickness will be larger for values of $\kappa$ close to zero or close to the bifurcation. The flattening is due to stronger contribution of the interaction potential.
        }
        \label{fig:three_nodes_complete}
\end{figure}

\begin{figure}[ht]
     \centering
     \begin{subfigure}[b]{0.45\textwidth}
         \centering
         \includegraphics[width=\textwidth]{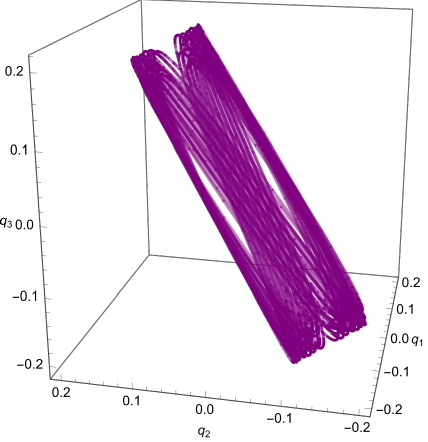}
         \caption{Regular 
         }
         \label{fig:trajectory_3_path_reg}
     \end{subfigure}
     \begin{subfigure}[b]{0.45\textwidth}
         \centering
         \includegraphics[width=\textwidth]{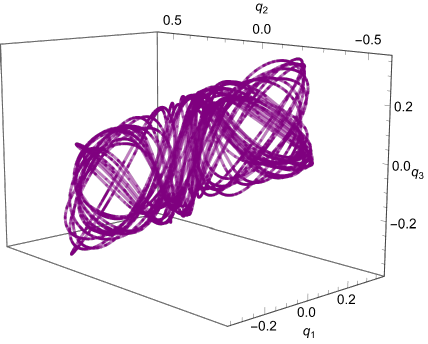}
         \caption{Chaotic 
         }
         \label{fig:trajectory_3_path_chaos_one}
     \end{subfigure}
     \begin{subfigure}[b]{0.45\textwidth}
         \centering
         \includegraphics[width=\textwidth]{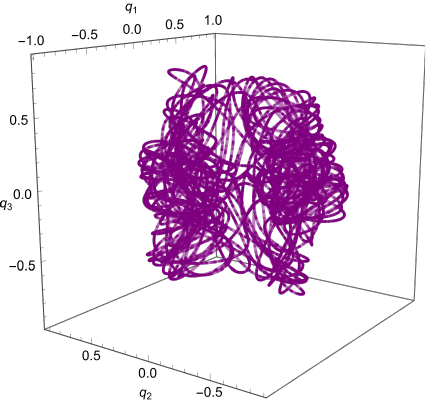}
         \caption{Chaotic 
         }
         \label{fig:trajectory_3_path_chaos_two}
     \end{subfigure}
        \caption{The three orbits shown have the same initial conditions $(q_1,q_2,q_3,p_1,p_2,p_3)=(1/5,1/10,1/7,0,0,0)$. 
        The figures represent the position's trajectory of the system. The critical parameters for the bifurcations are $\kappa^\text{crit}=1/6$ and $\kappa^\text{crit}=1/2$, the simulations are done for $\kappa=1/8$, Figure \ref{fig:trajectory_3_path_reg}, $\kappa=1/4$, Figure \ref{fig:trajectory_3_path_chaos_one}, and $\kappa=1$, Figure \ref{fig:trajectory_3_path_chaos_two}. In this case, the two bifurcations of the origin take place at different values of the coupling strength. The Lyapunov exponents, Table \ref{tab:lce}, display a correspondence between the bifurcation and the chaotic nature of the orbit. Indeed, after each of the bifurcations one positive exponent reach a considerably higher magnitude.
        }
        \label{fig:three_nodes_path}
\end{figure}
\section{Conclusions}

We characterised properties of synchrony and anti-synchrony spaces such as stability and bifurcations for a system of coupled pendula with non-negative interaction. Consequently, considering the oscillatory regime, i.e., low energies, we inspected the transitions to chaos of orbits close to the origin. The behaviour observed showcases a mechanism for the appearance of chaotic oscillations in conservative systems. For Hamiltonians with two degrees of freedom, it has been shown how chaos can emerge near a saddle-centre equilibrium possessing a homoclinic orbit \cite{mielke1992cascades}. In our paper, we have seen that higher-dimensional systems having mixed equilibria, i.e., generalisations of a saddle-center where the eigenvalues are only imaginary and reals, also show chaotic behaviour close to those points. Notice that, for the systems we studied, the homoclinic orbits appear, for example, in the invariant anti-synchrony spaces. Moreover, the combination of saddle and center behaviour is generic for the origin after bifurcation.

The relevance of studying bifurcations at low energies comes from the potential applications of the model studied. Hamiltonian functions that are bounded from below are a cornerstone for classical physical systems. Moreover, the ground state, or in other words the state with minimal energy, is of extreme relevance for quantum systems as well \cite{gozzi1983ground, spohn1989ground, nielsen2006cluster}. The model studied could also lead to interesting connections with breather solutions \cite{bambusi1996exponential, mackay2002effective, flach2008discrete, chernyavsky2016breathers}, where one could exploit in some detail the complex network structure and the relation to chaos.

Of course, the techniques we applied and developed can be adapted also to other nonlinear oscillators that are not necessarily pendula. Nevertheless, this is a potentially interesting research direction as it might be possible to identify interesting relations between the internal dynamics and the coupling. Also, different forms of interactions, beyond the diffusive form, can be explored in the future. For example, cross-term interactions lead in general to adjacency structures rather than Laplacian ones. 
Concerning detailed numerical results, we would like to point out that Hamiltonian systems present degeneracies that can even be enhanced by the potential symmetries of the network. As a consequence, a precise numerical bifurcation analysis is a non-trivial task. The starting point could be to use continuation software, e.g., MatCont, to explore, for instance, small network structures to see if relevant information can be extrapolated.

\vspace{6pt} 





\bibliography{bibliography.bib}
\bibliographystyle{abbrv}

\section*{Statements and Declarations}

\begin{description}
    \item[\textbf{Funding}] CK would like to thank the German Science Foundation (DFG) for support via a Sachbeihilfe grant (project number 444753754).
    RB and HJK declare that no funds, grants, or other support were received during the preparation of this manuscript.
    \item[\textbf{Conflict of Interest}] The authors have no relevant financial or non-financial interests to disclose.
\end{description}

\end{document}